\documentclass[12pt]{amsart}
\usepackage{geometry}
\usepackage[mathscr]{euscript}
\usepackage{ amssymb }
\usepackage[comma, numbers]{natbib}
\usepackage{bm}
\usepackage{enumerate}
\usepackage[english]{babel}
\usepackage{commath}
\usepackage{tikz-cd}
\usepackage{graphicx}
\usepackage [autostyle, english = american]{csquotes}
\MakeOuterQuote{"}
\usepackage{centernot}

\theoremstyle{plain}
\newtheorem{theorem}{Theorem}[section]

\theoremstyle{plain}
\newtheorem{lemma}[theorem]{Lemma}

\theoremstyle{plain}
\newtheorem{corollary}[theorem]{Corollary}

\theoremstyle{definition}
\newtheorem{definition}[theorem]{Definition}

\theoremstyle{plain}
\newtheorem{proposition}[theorem]{Proposition}

\theoremstyle{remark}

\theoremstyle{definition}
\newtheorem{example}[theorem]{Example}

\theoremstyle{plain}

\theoremstyle{plain}

\theoremstyle{plain}

\title{Coproducts of proximity spaces}
\author{Pawel Grzegrzolka}
\address{Stanford University, California, USA}
\email{pawelg@stanford.edu}

\keywords{proximity spaces, coproducts, Smirnov compactification, Stone-Cech compactification, minimal uniform compactification, metrizability, connectedness, dimension\\This is a pre-print of an article published in \textit{Afrika Matematika}. The final authenticated version is available online at: https://doi.org/10.1007/s13370-019-00757-1}

\begin{document}

\begin{abstract}
In this paper, we introduce coproducts of proximity spaces. After exploring several of their basic properties, we show that given a collection of proximity spaces, the coproduct of their Smirnov compactifications proximally and densely embeds in the Smirnov compactification of the coproduct of the original proximity spaces. We also show that the dense proximity embedding is a proximity isomorphism if and only if the index set is finite. After constructing a number of examples of coproducts and their Smirnov compactifications, we explore several properties of the Smirnov compactification of the coproduct, including its metrizability, connectedness of the boundary, dimension, and its relation to the Stone-Cech compactification. In particular, we show that the Smirnov compactification of the infinite coproduct is never metrizable and that its boundary is highly disconnected. We also show that the proximity dimension of the Smirnov compactification of the coproduct equals the supremum of the covering dimensions of the individual Smirnov compactifications and that the Smirnov compactification of the coproduct is homeomorphic to the Stone-Cech compactification if and only if each individual proximity space is equipped with the Stone-Cech proximity. We finish with an example of a coproduct with the covering dimension $0$ but the proximity dimension $\infty.$
\end{abstract}

\maketitle
\tableofcontents

\section{Introduction}
Given a metric space $(X,d)$, one can declare two subsets $A,B \subseteq X$ to be "close" if and only if $D(A,B)=0,$ where $D(A,B):=\inf\{d(a,b) \mid a\in A, b \in B\}.$ This relation exhibits certain interesting properties, such as 
\begin{itemize}
\item the closure of a subset $A\subseteq X$ is the collection of all points of $X$ that are close to $A,$
\item any function between metric spaces that sends close sets to close sets is continuous.
\end{itemize}
In fact, one can show that a function between two metric spaces is uniformly continuous if and only if the images of close sets are close. Proximity spaces (for the definition, see Section \ref{preliminaries}) were introduced by Efremovic (\cite{Efremovic1}, \cite{Efremovic2}) to axiomatize this intuitive notion of closeness in metric spaces. Since their introduction, many results regarding proximity spaces, their generalizations, and their relations to other structures such as topologies and uniformities have been presented. For an overview of many such results as well as an introduction to the theory of proximity spaces, the reader is reffered to \cite{powerful} and \cite{proximityspaces}. In particular, it is worth noting that proximities completely classify Haudorff compactifications of topological spaces and consequently serve as a powerful tool in studying compactifications (see Section \ref{preliminaries}). 

Despite the fact that the theory of proximities and their generalizations is an active area of research with applications to topology (\cite{set_open}, \cite{multiset},  \cite{hausdorff}, \cite{rough}, \cite{apartness}), analysis (\cite{Naimpally}, \cite{analysis1}), computer science and point-free geometry (\cite{point_free}, \cite{computer_science}, \cite{Peters}, \cite{computer}), boolean algebras (\cite{boolean}, \cite{Fedorchuk}), and large-scale geometry (\cite{paper3}, \cite{paper1}, \cite{paper2}, \cite{Honari}, \cite{Protasov}), to the best of author's knowledge the literature lacks the study of coproducts of proximity spaces. This paper is designed to develop the basic theory of coproducts of proximity spaces. After briefly discussing elementary and intuitive properties of such coproducts, we focus our attention on the properties of the unique Smirnov compactifications of such coproducts. In particular, we explore their metrizability, dimension, connectedness of the boundary, and their relation to other compactifications such as the Stone-Cech and the minimal uniform compactification.

In Section \ref{preliminaries}, we introduce the necessary background surrounding proximity spaces and their Smirnov compactifications. In Section \ref{coproduct_section}, we introduce coproducts of proximity spaces, and we explore several of their basic properties. In particular, we show that the coproduct of proximity spaces is indeed the coproduct in the category of proximity spaces. In section \ref{Smirnov_section}, we show that given a collection of proximity spaces, the coproduct of their Smirnov compactifications proximally and densely embeds in the Smirnov compactification of the coproduct of the original proximity spaces. We also show that the dense proximity embedding is a proximity isomorphism if and only if the index set is finite. To aid the understanding of coproducts and their Smirnov compactifications, in section \ref{examples_section} we compute several Smirnov compactifications of coproducts of different collections of proximity spaces. In particular, we give an example of a metrizable coproduct whose Smirnov compactification is strictly smaller than the Stone-Cech compactification, but strictly larger than the  minimal uniform compactification associated to the given metric. In section \ref{properties_of_the_coproduct} and subsequent subsections, we explore properties of the Smirnov compactifications of the coproduct. Among other results, we show that
\begin{itemize}
\item the Smirnov compactification of the coproduct is homeomorphic to the Stone-Cech compactification if and only if each individual proximity space is equipped with the Stone-Cech proximity (Theorem \ref{theorem2}),
\item the boundary of the Smirnov compactification of an infinite coproduct always contains the Stone-Cech corona of the naturals, and consequently the Smirnov compactification of an infinite coproduct is not metrizable (Theorem \ref{theorem123}),
\item the boundary of the Smirnov compactification of the coproduct is highly disconnected (Theorem \ref{connectedness_theorem})
\item the proximity dimension of the coproduct and its Smirnov compactification agrees with the covering dimension of the Smirnov compactification and the supremum of the covering dimensions of the individual Smirnov compactifications (Theorem \ref{last_theorem}).
\end{itemize}
Finally, we utilize coproducts to finish with an example of a proximity space with the covering dimension $0$ but the proximity dimension $\infty.$

\section{Preliminaries}\label{preliminaries}

In this section, we recall basic definitions and theorems regarding proximity spaces. The definitions and theorems in this section come from \cite{powerful} and \cite{proximityspaces}. A reader familiar with proximity spaces and Smirnov compactifications may want to skip this section and refer back to it when necessary. For more on proximity spaces, the reader is referred to \cite{powerful} and \cite{proximityspaces}.

\begin{definition}
Let $X$ be a set. A binary relation $\delta$ on the power set of $X$ is called a {\bf proximity} on $X$ if it satisfies the following axioms for all $A,B,C \subseteq X$:
\begin{enumerate}
	\item $A\delta B\implies B\delta A,$
	\item $A\delta B\implies A,B\neq\emptyset,$
	\item $A \cap B\neq\emptyset\implies A\delta B,$
	\item $A\delta (B\cup C)\iff A\delta B\text{ or }A\delta C,$
	\item $A {\centernot{\delta}}B\implies\exists E\subseteq X,\,A{\centernot{\delta}}E\text{ and }(X\setminus E){\centernot{\delta}}B,$
\end{enumerate}
where $A{\centernot{\delta}}B$ means that the statement "$A\delta B$" does not hold. If $A \delta B,$ then we say that $A$ is \textbf{close} to $B.$  A pair $(X,\delta)$ where $X$ is a set and $\delta$ is a proximity on $X$ is called a {\bf proximity space}.
\end{definition}

\begin{definition}
A proximity space $(X, \delta)$ is called \textbf{separated} if and only if
\[\{x\}\delta \{y\} \iff x=y\]
for all $x,y\in X.$
\end{definition}
Now let us introduce a few examples of proximity relations that we will use in this paper.

\begin{example}\label{subspace_proximity}
Let $(X, \delta)$ be a proximity space and let $Y\subseteq X.$ Then the relation $\delta_Y$ defined by
\[A\delta_Y B \iff A \delta B,\]
where $A$ and $B$ are subsets of $Y$, is a proximity relation on $Y,$ called the \textbf{subspace proximity} induced by $\delta.$
\end{example}

\begin{example}\label{discrete_proximity}
Let $X$ be a set. Then the relation $\delta$ defined by
\[A\delta B \iff A\cap B \neq \emptyset\]
is a separated proximity relation, called the \textbf{discrete proximity}
\end{example}

\begin{example}\label{metric_proximity}
Let $(X,d)$ be a metric space. Then the relation $\delta$ defined by
\[A \delta B \iff d(A,B)=0\]
is a separated proximity relation, called the \textbf{metric proximity} associated to the metric $d.$
\end{example}

\begin{example}\label{one_point_compactification}
Let $X$ be a locally compact Hausdorff space. Then the relation $\delta$ defined by
\[ A \delta B \iff A,B \text{ are not compact in } X \text{ or } \bar{A} \cap \bar{B} \neq \emptyset,\]
where $\bar{A}$ denotes the closure of $A$ in $X$, is a separated proximity relation, called the \textbf{Aleksandroff proximity}.
\end{example}

To understand the following example, recall that two subsets $A$ and $B$ of a topological space $X$ are called \textbf{functionally distinguishable} if and only if there exists a continuous function $f:X \to [0,1]$ such that $f(A)=0$ and $f(B)=1.$ If such a function does not exists, we say that $A$ and $B$ are \textbf{functionally indistinguishable}.

\begin{example}\label{Stone-Cech_proximity}
Let $X$ be a completely regular Hausdorff space. Then the relation $\delta$ defined by
\[A \delta B \iff A \text{ and } B \text{ are functionally indistinguishable}\]
is a separated proximity relation, called the \textbf{Stone-Cech proximity}.
\end{example}

\begin{example}\label{standard_proximity}
Let $X$ be a normal Hausdorff space. Then the relation $\delta$ defined by
\[A\delta B \iff \bar{A}\cap \bar{B} \neq \emptyset,\]
where $\bar{A}$ denotes the closure of $A$ in $X$, is a separated proximity relation, called the \textbf{standard proximity}.
\end{example}

Notice that the Stone-Cech proximity and the standard proximity coincide when the base space is normal and Hausdorff.

The category of proximity spaces consists of objects being proximity spaces and morphisms being proximity maps, as in the following definition.

\begin{definition}
Let $(X, \delta_X)$ and $(Y, \delta_Y)$ be two proximity spaces. Then a function $f:X \to Y$ is called a \textbf{proximity map} if and only if
\[A \delta_X B \implies f(A) \delta_Y f(B).\]
A bijective proximity map whose inverse is also a proximity map is called a \textbf{proximity isomorphism}. A proximity isomorphism onto a subspace of a proximity space in called a \textbf{proximity embedding}.
\end{definition}

Any proximity relation $\delta$ on $X$ induces a topology on $X$ defined by
\[A \text{ closed } \iff (A \delta x \implies x \in A).\]
That induced topology has many interesting properties. In particular, it is well-known that:
\begin{itemize}
\item the induced topology is always completely regular,
\item the induced topology is Hausdorff if and only if the proximity is separated,
\item all proximity maps are continuous in the induced topologies. Consequently, proximity isomorphisms are homeomorphisms and proximity embeddings are homeomorphic embeddings,
\item two different proximities can induce the same topology,
\item when the topology on $X$ is compact and Hausdorff, there is only one proximity inducing that topology, namely the standard proximity described in Example \ref{standard_proximity}.
\end{itemize}

It turns out that proximities serve as a powerful tool in studying Hausdorff compactifications (where by a \textbf{compactification} of a topological space $X$ we mean a compact space $\bar{X}$ such that $X$ densely embedds in $\bar{X}.$ In that case, we usually consider $X$ as a subspace of $\bar{X}$). In particular, every separated proximity space with its induced topology is a dense subspace of a unique (up to proximity isomorphism) compact and Hausdorff space, called the Smirnov compactification. In fact, given a completely regular Hausdorff space $X$, compatible proximities (i.e., proximities whose induced topology agrees with the original topology on $X$) are in a bijective correspondence with Hausdorff compactifications of $X.$ Let us recall the construction of the Smirnov compactification, as it will be useful to us in this paper.

\begin{definition}
	A {\bf cluster} in a separated proximity space $(X,\delta)$ is a nonempty collection $\sigma$ of subsets of $X$ satisfying the following:
	\begin{enumerate}
		\item For all $A,B\in\sigma$, $A\delta B,$
		\item If $C\delta A$ for all $A\in\sigma,$ then $C\in\sigma,$
		\item If $(A\cup B)\in\sigma,$ then either $A\in\sigma$ or $B\in\sigma.$
	\end{enumerate}
A cluster $\sigma$ is called a {\bf point cluster} if $\{x\}\in\sigma$ for some $x\in X.$
\end{definition}
Clusters will become "points" in the Smirnov compactification. They have many natural properties (some of them similar to properties of ultrafilters), for example
\begin{enumerate}
\item $A \in \sigma$ and $A \subseteq B \implies B \in \sigma,$
\item $A \in \sigma$ or $X\setminus A \in \sigma,$
\item $\sigma_1 \subseteq \sigma_2 \implies \sigma_1=\sigma_2.$
\end{enumerate}

The following two propositions show the relationship between clusters in a proximity space and clusters in any of its subspaces.

\begin{proposition}\label{generating_cluster}
If $(Y, \delta_Y)$ is a subspace of a proximity space $(X, \delta_X),$ then every cluster $\sigma_Y$ in $Y$ is a subset of a unique cluster $\sigma_X$ in $X$ defined by
\[\sigma_X:=\{A\subseteq X \mid A\delta_X B \text{ for all }B \in \sigma_Y\}.\]
In such a case, we say that $\sigma_Y$ \textbf{generates} $\sigma_X.$
\end{proposition}
\begin{proof}
See corollary 5.18 in \cite{proximityspaces}.
\end{proof}

\begin{proposition}\label{cluster_in_a_subspace}
If $(Y, \delta_Y)$ is a subspace of a proximity space $(X, \delta_X),$ and $\sigma_X$ is a cluster in $X$ such that $Y \in \sigma_X,$ then there exists a unique cluster $\sigma_Y$ in $Y$ contained in $\sigma_X,$ namely
\[\sigma_Y:=\{A \subseteq Y \mid A \in \sigma_X\}.\]
In fact, $\sigma_Y$ generates $\sigma_X.$
\begin{proof}
See Theorem 5.16 in \cite{proximityspaces}.
\end{proof}
\end{proposition}

The set of all clusters in a separated proximity space $X$ is usually denoted by $\mathfrak{X}$. Given a set $\mathcal{A}\subseteq\mathfrak{X}$ and a subset $C\subseteq X,$ we say that $C$ {\bf absorbs} $\mathcal{A}$ if $C\in\sigma$ for every $\sigma\in\mathcal{A}$.

\begin{theorem}\label{theorem11}
	Let $(X,\delta)$ be a separated proximity space and $\mathfrak{X}$ the corresponding set of clusters. The relation $\delta^{*}$ on the power set of $\mathfrak{X}$ defined for all $ \mathcal{A}, \mathcal{B}\ \subseteq \mathfrak{X}$ by
	\[\mathcal{A}\delta^{*}\mathcal{B}\iff A\delta B\]
	for all sets $A,B\subseteq X$ that absorb $\mathcal{A}$ and $\mathcal{B}$, respectively, is a proximity on $\mathfrak{X}.$ In fact, $\delta^*$ induces a compact and Hausdorff topology on $\mathfrak{X}$ into which $X$ proximally embeds as a dense subspace (by mapping each point to its corresponding point cluster). Also, $\mathfrak{X}$ is a unique (up to proximity isomorphism) compact Hausdorff space into which $X$ proximally and densely embedds.
\end{theorem}
\begin{proof}
See Theorem 7.7 in \cite{proximityspaces}.
\end{proof}

The compactification described in Theorem \ref{theorem11} is called the {\bf Smirnov compactification} of the proximity space $(X,\delta)$.

Since compact Hausdorff spaces have a unique seperated proximity inducing that topology (namely, the standard proximity), the above theorem implies that given a topological space $X$ and its Hausdorff compactification $\bar{X},$ the proximity 
$\delta$ on $X$ whose Smirnov compactification is homeomorphic to $\bar{X}$ is given by
\[A\delta B \iff cl_{\bar{X}}(A) \cap cl_{\bar{X}}(B) \neq \emptyset.\]
Using this fact, one can verify that the Aleksandroff proximity induces the one-point compactification and the Stone-Cech proximity induces the Stone-Cech compactification. The compactification induced by the metric proximity is called the \textbf{minimal uniform compactification} and is studied thoroughly in \cite{woods}. 

To finish this section, let us describe the bijective correspondence between compactifications and compatible proximites. It is well known that given a topological space $X,$ all compactifications of $X$ can be partially ordered by
\[\bar{X}_1 \geq \bar{X}_2 \iff \text{ the identity map on } X \text{ extends to a continuous map from  } \bar{X}_1 \text{ to } \bar{X}_2.\]
Similarly, one can partially order proximity relations on $X$ by
\[\delta_1 \geq \delta_2 \iff (A\delta_1 B \implies A \delta_2 B).\]
In that case, we say that $\delta_1$ is \textbf{finer} than $\delta_2.$
\begin{theorem}
Let $X$ be a completely regular Hausdorff space. Then there exists a bijective correspondence between proximities compatible with the topology on $X$ and compactifications of $X$ given by
\[\delta_1 \geq \delta_2 \iff \bar{X}_1 \geq \bar{X}_2,\]
where $\bar{X}_i$ is the Smirnov compactification of $X$ associated to the proximity $\delta_i.$
\end{theorem}
\begin{proof}
See Theorem 7.11 in \cite{proximityspaces}.
\end{proof}

\section{Definition and basic properties}\label{coproduct_section}
In this and the following sections, the notation $(X_{\alpha})_{\alpha}$ denotes a collection of objects indexed by an arbitrary (finite or countable or uncountable) set $I.$ Elements of $I$ are usually denoted by the Greek letters $\alpha, \beta$ etc. Also, given a disjoint union of sets $\coprod_{\alpha}X_{\alpha}$ and $A \subseteq \coprod_{\alpha}X_{\alpha},$ we define
\[A_{\alpha}:=A\cap X_{\alpha}.\]

\begin{definition}\label{coproduct_definition}
Let $(X_{\alpha}, \delta_{\alpha})_{\alpha}$ be a collection of proximity spaces. Then the \textbf{coproduct} of $(X_{\alpha}, \delta_{\alpha})_{\alpha \in I}$ is the proximity space 
\[(\coprod_{\alpha} X_{\alpha}, \delta),\]
where 
$\coprod_{\alpha} X_{\alpha}$ is the disjoint union of sets $(X_{\alpha})_{\alpha},$ and $\delta$ is defined by
\[A\delta B \iff A_{\alpha} \delta_{\alpha} B_{\alpha} \text{ for some } \alpha.\]
\end{definition}

\begin{proposition}
Let $(X_{\alpha}, \delta_{\alpha})_{\alpha}$ be a collection of proximity spaces. Then $(\coprod_{\alpha} X_{\alpha}, \delta)$ is a proximity space. 
\end{proposition}
\begin{proof}
Clearly $\delta$ satisfies axioms 1,2, and 3 of a proximity space. Notice that 
\begin{equation*}
\begin{split}
(A\cup B)\delta C & \iff ((A\cup B)_\alpha \delta_{\alpha} C_{\alpha} \text{ for some } \alpha \\
 &  \iff ((A_{\alpha}\cup B_{\alpha}) \delta_{\alpha} C_{\alpha} \text{ for some } \alpha \\
 & \iff A_{\alpha} \delta_{\alpha} C_{\alpha} \text{ or } B_{\alpha} \delta_{\alpha} C_{\alpha} \text{ for some } \alpha \\
 & \iff A \delta C \text{ or } B \delta C,
\end{split}
\end{equation*}
which shows axiom $4.$ Finally, to see axiom $5,$ assume that $A{\centernot{\delta}}B.$ This means that $A_{\alpha}{\centernot{\delta}}_{\alpha}B_{\alpha}$ for all $\alpha.$ Consequently, for each $\alpha$ there exists $E_{\alpha}$ such that 
\[A_{\alpha} {\centernot{\delta}}_{\alpha} E_{\alpha} \text{ and } (X_{\alpha}\setminus E_{\alpha}){\centernot{\delta}}_{\alpha}B_{\alpha}.\]
Thus, $E:= \bigcup_{\alpha} E_{\alpha}$ is the desired set in $\coprod_{\alpha} X_{\alpha}$ such that
\[A {\centernot{\delta}} E \text{ and } (\coprod_{\alpha}X_{\alpha} \setminus E){\centernot{{\delta}}}B. \]
\end{proof}

The following three propositions are immediate consequences of the definition of $\delta,$ and thus the proofs are left as easy exercises.

\begin{proposition}
Let $(X_{\alpha}, \delta_{\alpha})_{\alpha}$ be a collection of proximity spaces and let $(\coprod_{\alpha} X_{\alpha}, \delta)$ be the associated coproduct. Then for any $\alpha,$ we have that $\delta_{\alpha}$ is the subspace proximity on $X_{\alpha}$ induced by $\delta.$ \qed
\end{proposition}

\begin{proposition}
Let $(X_{\alpha}, \delta_{\alpha})_{\alpha}$ be a collection of proximity spaces and let $(\coprod_{\alpha} X_{\alpha}, \delta)$ be the associated coproduct. Then $\delta$ is separated if and only if $\delta_{\alpha}$ is separated for every $\alpha.$ \qed
\end{proposition}

\begin{proposition}
Let $(X_{\alpha}, \delta_{\alpha})_{\alpha}$ be a collection of proximity spaces and let $(\coprod_{\alpha} X_{\alpha}, \delta)$ be the associated coproduct. Let $(Y, \delta_Y)$ be a proximity space, and for any $\alpha,$ let $f_{\alpha}:X_{\alpha} \to Y$ be a map. Then the map $f: (\coprod_{\alpha} X_{\alpha}, \delta) \to Y$ that agrees with $f_{\alpha}$ on $X_{\alpha}$ for all $\alpha$ is a proximity map if and only if $f_{\alpha}$ is a proximity map for all $\alpha.$ \qed
\end{proposition}

The following proposition shows that the topology on the coproduct is the disjoint union topology.

\begin{proposition}
Let $(X_{\alpha}, \delta_{\alpha})_{\alpha}$ be a collection of proximity spaces and let $(\coprod_{\alpha} X_{\alpha}, \delta)$ be the associated coproduct. Then the topology on $\coprod_{\alpha} X_{\alpha}$ induced by $\delta$ coincides with the disjoint union topology on $\coprod_{\alpha} X_{\alpha}$ coming from individual topologies on each $X_{\alpha}.$
\end{proposition}
\begin{proof}
Let $A\subseteq \coprod_{\alpha} X_{\alpha}$ be closed in the topology induced by $\delta.$ Let $\alpha$ be arbitrary and consider $A_{\alpha}.$ To see that $A_{\alpha}$ is closed in the topology induced by $\delta_{\alpha},$ let $x \in X_{\alpha}$ be such that $x\delta_{\alpha} A_{\alpha}.$ In particular, $x \delta A,$ which implies that $x \in A_{\alpha}.$ Thus, $A_{\alpha}$ is closed in $X_{\alpha}.$ Since $\alpha$ was arbitrary, $A_{\alpha}$ is closed in $X_{\alpha}$ for any $\alpha.$ Thus, $A$ is closed in the disjoint union topology.

Conversely, let $A\subseteq \coprod_{\alpha} X_{\alpha}$ be closed in the disjoint union topology. Let $x \in \coprod_{\alpha} X_{\alpha}$ be such that $x \delta A.$ In particular, $x \in X_{\alpha}$ for some $\alpha,$ and thus $x \delta_{\alpha} A_{\alpha}.$ Since $A_{\alpha}$ is closed in $X_{\alpha},$ this means that $x \in A_{\alpha},$ and consequently $x \in A,$ showing that $A$ is closed in the topology induced by $\delta.$
\end{proof}

The following proposition shows that the coproduct of proximity spaces is really the coproduct in the category of proximity spaces.

\begin{proposition}
Let $(X_{\alpha}, \delta_{\alpha})_{\alpha}$ be a collection of proximity spaces and for any $\alpha,$ let $i_{\alpha}: X_{\alpha} \to \coprod_{\alpha} X_{\alpha}$ be the canonical set injection. Then $(\coprod_{\alpha} X_{\alpha}, \delta)$ is the coproduct in the category of proximity spaces, i.e., if $(Y, \delta_Y)$ is a proximity space and for each $\alpha$ we have a proximity map $f_{\alpha}: X_{\alpha} \to Y$, then there exists a unique proximity map $h: \coprod_{\alpha} X_{\alpha} \to Y$ such that $h \circ i_{\alpha} =f_{\alpha}$ for any $\alpha.$ In other words, the following diagram commutes for any $\alpha$:
\begin{center}
\begin{tikzcd}[column sep=huge, row sep =huge]
 & Y\\
X_{\alpha} \arrow{r}{i_{\alpha}} \arrow{ru}{f_{\alpha}} & \coprod_{\alpha} X_{\alpha} \arrow[dashrightarrow]{u}{h}
\end{tikzcd}
\end{center}
\end{proposition}
\begin{proof}
Clearly each $i_{\alpha}$ is a proximity map. Define $h: \coprod_{\alpha} X_{\alpha} \to Y$ so that it agrees with $f_{\alpha}$ on $X_{\alpha}$ for all $\alpha.$ That is, given $x\in \coprod_{\alpha} X_{\alpha},$ we know that $x\in X_{\alpha}$ for a unique ${\alpha},$ and consequently we set $h(x)=f_{\alpha}(x).$ It is then clear that $h \circ i_{\alpha} =f_{\alpha}$ for any $\alpha,$ and that $h$ is the unique function satisfying that condition. Notice that $h$ is a proximity map. For if $A,B \subseteq \coprod_{\alpha} X_{\alpha}$ such that $A\delta B,$ then $A_{\alpha} \delta_{\alpha} B_{\alpha}$ for some $\alpha.$ Consequently, $f_{\alpha}(A_{\alpha})\delta_Y f_{\alpha}(B_{\alpha}).$ Since $f_{\alpha}(A_{\alpha})\subseteq h(A)$ and $f_{\alpha}(B_{\alpha})\subseteq h(B),$ we have that $h(A) \delta_Y h(B).$
\end{proof}

\section{Coproduct and the Smirnov compactification}\label{Smirnov_section}

Given a collection of proximity spaces, one can
\begin{enumerate}
\item take their Smirnov compactifications, and then form the coproduct of the resulting collection of compact proximity spaces,
\item form the coproduct of the original collection, and then take the Smirnov compactification of the resulting proximity space. 
\end{enumerate}
The following proposition shows that the space obtained in $1$ proximally and densely embedds in the space obtained in $2.$ Also, it shows that the two spaces agree if and only if the index set is finite, i.e., the construction of the Smirnov compactification commutes with the coproduct if and only if the index set is finite.

\begin{theorem}\label{theorem1}
Let $(X_{\alpha}, \delta_{\alpha})_{\alpha}$ be a collection of separated proximity spaces and let $(\coprod_{\alpha} X_{\alpha}, \delta)$ be the associated coproduct. For any $\alpha,$ let $(\mathfrak{X}_{\alpha}, \delta_{\alpha}^*)$ denote the Smirnov compactification of $(X_{\alpha}, \delta_{\alpha}).$ Let $(\overline{\coprod_{\alpha} X_{\alpha}}, 
\delta^*)$ denote the Smirnov compactification of the coproduct $(\coprod_{\alpha} X_{\alpha}, \delta).$ Then 
\[(\coprod_{\alpha} \mathfrak{X}_{\alpha}, \delta_1)\hookrightarrow (\overline{\coprod_{\alpha} X_{\alpha}}, \delta^*),\]
where $\delta_1$ is the proximity associated to the coproduct of the collection of proximity spaces $(\mathfrak{X}_{\alpha}, \delta_{\alpha}^*)_{\alpha},$ and $\hookrightarrow$ denotes a dense proximity embedding. What is more, the dense proximity embedding is a proximity isomorphism if and only if the index set is finite. In other words, the coproduct of Smirnov compactifications is the Smirnov compactification of the coproduct if and only if the index set is finite.
\end{theorem}
\begin{proof}
For any $\alpha,$ we have that $\delta_{\alpha}$ is a subspace proximity on $X_{\alpha}$ induced by $\delta.$  Thus, by Proposition \ref{generating_cluster}, we know that every cluster $\sigma_{\alpha}$ in $X_{\alpha}$ generates a unique cluster $\sigma$ in $\coprod_{\alpha} X_{\alpha}$ defined by
\[\sigma:=\{A \subseteq \coprod_{\alpha} X_{\alpha} \mid A\delta B \text{ for every } B \in \sigma_{\alpha}\}.\]
In particular, $\sigma_{\alpha} \subseteq \sigma.$ Define $F: \coprod_{\alpha} \mathfrak{X}_{\alpha}\to \overline{\coprod_{\alpha} X_{\alpha}}$ by sending a a cluster in $X_{\alpha}$ for some $\alpha$ (i.e., a point in $\coprod_{\alpha} \mathfrak{X}_{\alpha}$) to the unique cluster it generates in $\coprod_{\alpha} X_{\alpha}$ (i.e., a point in $\overline{\coprod_{\alpha} X_{\alpha}}$). To see that the map $F$ is injective, for contradiction assume that $F(\sigma_{\alpha})=F(\sigma_{\beta})=\sigma$ for some $\sigma.$ If $\sigma_{\alpha}$ and $\sigma_{\beta}$ are clusters in different spaces $X_{\alpha}$ and $X_{\beta},$ respectively, then $X_{\alpha} \in \sigma_{\alpha}\subseteq \sigma$ and $X_{\beta} \in \sigma_{\beta}\subseteq \sigma,$ but $X_{\alpha} {\centernot{\delta}} X_{\beta},$ a contradiction to the definition of a cluster. Thus, $\sigma_{\alpha}$ and $\sigma_{\beta}$ are clusters in the same space $X_{\alpha}.$ However, since $X_{\alpha} \in \sigma,$ by Proposition \ref{cluster_in_a_subspace} there can be only one cluster in $X_{\alpha}$ generating $\sigma.$ Thus, it has to be that $\sigma_{\alpha}=\sigma_{\beta},$ which proves injectivity. 

To see that $F$ is a proximity mapping, let $\mathcal{A} \delta_1 \mathcal{B},$ where $\mathcal{A},\mathcal{B} \subseteq \coprod_{\alpha} \mathfrak{X}_{\alpha}.$ In particular, this means that there exists an $\alpha$ such that $\mathcal{A}_{\alpha}\delta_{\alpha}^* \mathcal{B}_{\alpha}.$ Thus, for any subset $A_{\alpha}\subseteq X_{\alpha}$ that absorbs $\mathcal{A}_{\alpha}$ and any subset $B_{\alpha}\subseteq X_{\alpha}$ that absorbs $\mathcal{B}_{\alpha},$ we have that $A_{\alpha} \delta_{\alpha} B_{\alpha}.$ For contradiction, assume that  $F(\mathcal{A}){\centernot{\delta}}^* F(\mathcal{B}).$ This means that there exists a subset $A\subseteq \coprod_{\alpha} X_{\alpha}$ that absorbs $F(\mathcal{A})$ and a subset $B\subseteq \coprod_{\alpha} X_{\alpha}$ that absorbs $F(\mathcal{B})$ such that $A{\centernot{\delta}}B.$ In particular, $A_{\alpha} {\centernot{\delta}}_{\alpha} B_{\alpha}.$ Thus, it is enough to show that $A_{\alpha}$ absorbs $\mathcal{A}_{\alpha}$ and $B_{\alpha}$ absorbs $\mathcal{B}_{\alpha},$ since this will contradict $\mathcal{A}_{\alpha}\delta_{\alpha}^* \mathcal{B}_{\alpha}.$ For contradiction, assume that $A_{\alpha}$ does not absorb $\mathcal{A}_{\alpha}.$ This means that there exists a cluster $\sigma_{\alpha}$ in $\mathcal{A}_{\alpha}$ such that $A_{\alpha} \notin \sigma_{\alpha}.$ Notice that this also means that $A_{\alpha} \notin \sigma:=F(\sigma_{\alpha})$ (because by Proposition \ref{cluster_in_a_subspace},  we know that $\sigma$ and $\sigma_{\alpha}$ agree on $X_{\alpha}$). This shows that $A \notin \sigma,$ since 
\[A=A_{\alpha} \cup \big((\coprod_{\alpha}X_{\alpha})\setminus X_{\alpha}\big),\]
but $A_{\alpha} \notin \sigma$ and $\big((\coprod_{\alpha}X_{\alpha})\setminus X_{\alpha}\big) \notin \sigma$ (where the latter non-containment follows from the fact that $X_{\alpha} \in \sigma_{\alpha} \subseteq \sigma$ and the fact that $X_{\alpha}{\centernot{\delta}} \big((\coprod_{\alpha}X_{\alpha})\setminus X_{\alpha}\big)$). This contradicts the fact that $A$ absorbs $F(\mathcal{A}).$ Thus, $A_{\alpha}$ absorbs $\mathcal{A}_{\alpha}.$ Similarly, one can show that $B_{\alpha}$ absorbs $\mathcal{B}_{\alpha},$ which finishes the proof that $F$ is a proximity map.

To see that the inverse map (with the domain of the inverse map being the image of $F$) is also a proximity map, assume that $F(\mathcal{A})\delta^* F(\mathcal{B}),$ and for contradiction also assume that $\mathcal{A}{\centernot{\delta}}_1\mathcal{B}.$ This means that for all $\alpha,$ there exists a subset $A_{\alpha}\subseteq X_{\alpha}$ that absorbs $\mathcal{A}_{\alpha}$ and a subset $B_{\alpha}\subseteq X_{\alpha}$ that absorbs $\mathcal{B}_{\alpha}$ such that $A_{\alpha}{\centernot{\delta}}_{\alpha}B_{\alpha}.$ Let $A:=\cup_{\alpha}A_{\alpha}$ and $B:=\cup_{\alpha}B_{\alpha}.$ We claim that $A$ and $B$ are subsets of $\coprod_{\alpha}X_{\alpha}$ that absorb $F(\mathcal{A})$ and $F(\mathcal{B}),$ respectively, but $A{\centernot{\delta}}B.$ This will contradict $F(\mathcal{A})\delta^* F(\mathcal{B}),$ and finish the proof that $\mathcal{A}\delta_1\mathcal{B}.$ Since for all $\alpha,$ we have that $A_{\alpha}{\centernot{\delta}}_{\alpha}B_{\alpha},$ it is clear that $A{\centernot{\delta}}B.$ Let us show that $A$ absorbs $F(\mathcal{A}).$ For contradiction, assume that $A$ does not absorb $F(\mathcal{A}),$ i.e., there exists $\sigma_{\alpha}$ in some $X_{\alpha}$ such that $A\notin \sigma:=F(\sigma_{\alpha}).$ In particular, $A_{\alpha} \notin \sigma_{\alpha}$ (it is because $A_{\alpha} \in \sigma_{\alpha}$ would imply $A_{\alpha} \in \sigma,$ and consequently $A \in \sigma$). But this contradicts the fact that $A_{\alpha}$ absorbs $\sigma_{\alpha}.$ Thus, it has to be that $A$ absorbs $F(\mathcal{A}).$ Similarly one can show that $B$ absorbs $F(\mathcal{B}).$ This finishes that proof that $F$ is a proximity embedding. Consequently, we can think of $\coprod_{\alpha} \mathfrak{X}_{\alpha}$ as a subspace of $\overline{\coprod_{\alpha} X_{\alpha}}.$

The density of $F$  follows from the fact that 
\[\coprod_{\alpha} X_{\alpha} \subseteq \coprod_{\alpha} \mathfrak{X}_{\alpha}\subseteq \overline{\coprod_{\alpha} X_{\alpha}},\]
and the fact that any proximity space in dense in its Smirnov compactification (i.e., $\coprod_{\alpha} X_{\alpha}$ is dense in $\overline{\coprod_{\alpha} X_{\alpha}}$).

It is clear that the dense proximity embedding from the above theorem is not a proximity isomorphism when the index set is infinite, since the induced topology on $(\coprod_{\alpha} \mathfrak{X}_{\alpha}, \delta_1)$ is not compact, whereas the induced topology on $(\overline{\coprod_{\alpha} X_{\alpha}}, \delta^*)$ is compact.
To see the converse, assume that the index set is finite. To show that $F$ is a proximity isomorphism, it is enough to show that $F$ is surjective. Let $\sigma$ be a cluster in $\coprod_{\alpha} X_{\alpha}.$ Since  $\coprod_{\alpha} X_{\alpha} \in \sigma$ and the index set is finite, by the third property of a cluster it has to be that $X_{\alpha} \in \sigma$ for some $\alpha.$ Consequently, by Proposition \ref{cluster_in_a_subspace} there exists a cluster $\sigma_{\alpha}$ in $X_{\alpha}$ that generates $\sigma,$ i.e., $F(\sigma_{\alpha})=\sigma.$ Thus, $F$ is surjective.
\end{proof}

In fact, when the index set is infinite, it is easy to construct points in $\overline{\coprod_{\alpha} X_{\alpha}}$ that are not images of any points in $\coprod_{\alpha} \mathfrak{X}_{\alpha},$ i.e., one can easily construct a cluster in $\coprod_{\alpha} X_{\alpha}$ that is not generated by any cluster in $X_{\alpha}$ for any $\alpha,$ as the following example shows.
\begin{example}
With the notation from Theorem \ref{theorem1} (and assuming that the index set is infinite), define
\[L:=\{A \subseteq \coprod_{\alpha} X_{\alpha} \mid X_{\alpha} \subseteq A \text{ for all but finitely many } \alpha\}.\]
It is easy to see that $L$ is a filter on the set $\coprod_{\alpha} X_{\alpha}.$ By Theorem 5.8 in \cite{proximityspaces}, there exists a cluster $\sigma$ in $\coprod_{\alpha} X_{\alpha}$ such that
\[\sigma:=\{A \subseteq \coprod_{\alpha} X_{\alpha} \mid A \delta B \text{ for every } B\in L\}.\]
In particular, $L \subseteq \sigma.$ Notice that $\sigma$ cannot be generated by any cluster $\sigma_{\alpha}$ in $X_{\alpha}$ for any $\alpha.$ For if $\sigma$ is generated by some cluster $\sigma_{\alpha}$ in $X_{\alpha}$ for some $\alpha,$ then $X_{\alpha} \in \sigma_{\alpha} \subseteq \sigma$ and $\Big((\coprod_{\alpha} X_{\alpha})\setminus X_{\alpha}\Big) \in L \subseteq \sigma,$ but 
\[\Big((\coprod_{\alpha} X_{\alpha})\setminus X_{\alpha}\Big){\centernot{\delta}} X_{\alpha},\]
a contradiction.
\end{example}

With the notation from Theorem \ref{theorem1}, points in $\coprod_{\alpha} \mathfrak{X}_{\alpha}$ can be identified with points in $\overline{\coprod_{\alpha}X_{\alpha}}.$ In particular, any cluster $\sigma_{\alpha}$ in $X_{\alpha}$ for some $\alpha$ (which is a collection of subsets of $X_{\alpha}$) is identified with the cluster $\sigma$ it generates in $\coprod_{\alpha}X_{\alpha}$ (which is a collection of subsets of $\coprod_{\alpha}X_{\alpha}$). In particular,
\[\sigma=\{A \subseteq \coprod_{\alpha}X_{\alpha} \mid A\delta B \text{ for all } B \in \sigma_{\alpha}\}.\]
But by the definition of $\delta,$ it is easy to see that this is equivalent to 
\[\sigma=\{A \subseteq \coprod_{\alpha}X_{\alpha} \mid \text{ there exists } B \in \sigma_{\lambda} \text{ such that } B \subseteq A\},\]
i.e., $\sigma$ consists of supersets of elements of $\sigma_{\alpha}.$

\begin{corollary}
With the notation from Theorem \ref{theorem1}, $(\overline{\coprod_{\alpha} X_{\alpha}}, \delta^*)$ is the Smirnov compactification of  $(\coprod_{\alpha} \mathfrak{X}_{\alpha}, \delta_1).$ In particular, $\mathcal{A},\mathcal{B} \subseteq \coprod_{\alpha} \mathfrak{X}_{\alpha}$ are $\delta_1$ close if and only if their closures in $\overline{\coprod_{\alpha}X_{\alpha}}$ intersect.
\end{corollary}
\begin{proof}
The first statement follows from Theorem \ref{theorem1} and the fact that the Smirnov compactification of a proximity space is unique. The second statement follows from the fact that compact Hausdorff spaces have only one compatible proximity defined by: two subsets are close if and only if their closures intersect.
\end{proof}

\section{Examples}\label{examples_section}
The coproduct of a collection of proximity spaces has the associated proximity, and consequently the  unique Smirnov compactification. Before we explore properties of that compactification, in this section we provide several examples of the Smirnov compactifications of certain coproducts. In all of the following examples, the coproduct $(\coprod_{\alpha}X_{\alpha}, \delta)$ is always equipped with the topology induced by $\delta.$ Thus, the statements such as "the compactification of $\coprod_{\alpha}X_{\alpha}$" are unambiguous.

\begin{example}\label{example1}
Let $X_1:=\mathbb{R}$ be equipped with the standard proximity
\[A \delta_1 B \iff \bar{A} \cap \bar{B} \neq \emptyset,\]
let $X_2:=\mathbb{R}$ be equipped with the metric proximity
\[A \delta_2 B \iff d(A,B)=0,\]
and let $X_3:=\mathbb{R}$ be equipped with the Alexandroff proximity
\[ A \delta_3 B \iff A,B \text{ are not compact in } \mathbb{R} \text{ or } \bar{A} \cap \bar{B} \neq \emptyset.\]
All of these proximities induce the same Euclidean topology on $\mathbb{R}.$ Also, the Smirnov compactification of the coproduct of $X_1, X_2,$ and $X_3$ is homeomorphic to the disjoint union of the Stone-Cech compactification of $\mathbb{R},$ the minimum uniform compactification of $\mathbb{R},$ and the one-point compactification of $\mathbb{R}.$
\end{example}

\begin{example}\label{example2}
Let $(X_{\alpha}, \delta_{\alpha})_{\alpha}$ be an infinite collection of proximity spaces such that each proximity space $X_{\alpha}$ consists of a single point. Notice that two subsets of the coproduct $(\coprod_{\alpha}X_{\alpha}, \delta)$ are close if and only if they intersect. Thus, the coproduct proximity $\delta$ is the finest proximity compatible with the induced topology, and consequently the Smirnov compactification of the coproduct is homeomorphic to the Stone-Cech compactification of $\coprod_{\alpha}X_{\alpha}.$
\end{example}

\begin{example}\label{example3}
To generalize the above example, let $(X_{\alpha}, \delta_{\alpha})_{\alpha}$ be an infinite collection of proximity spaces such that each proximity space $X_{\alpha}$ is equipped with the discrete proximity $\delta_{\alpha},$ i.e., for any two subsets $A_{\alpha},B_{\alpha} \subseteq X_{\alpha},$ we have
\[A_{\alpha}\delta_{\alpha}B_{\alpha} \iff A\cap B \neq \emptyset.\]
Then any two subsets of the coproduct $(\coprod_{\alpha}X_{\alpha}, \delta)$ are close if and only if they intersect, and consequently the Smirnov compactification of the coproduct is homeomorphic to the Stone-Cech compactification of $\coprod_{\alpha}X_{\alpha}.$
\end{example}

In all of the above examples, the type of the compactification of the coproduct somehow agreed with the compactifications of the individual spaces. In Example \ref{example1}, the compactification of the coproduct was just the disjoint union of the individual compactifications. In Example \ref{example2} and Example \ref{example3}, the individual compactifications as well as the compactifications of the coproducts were the Stone-Cech compactifications. But that phenomenon does not always happen, as the following examples show.

\begin{example}
Let $(X_{\alpha})_{\alpha}$ be an infinite collection of locally compact Hausdorff spaces such that each space consists of at least $2$ points. Equip each $X_{\alpha}$ with the compatible Aleksandroff proximity $\delta_{\alpha}$ i.e., for any two subsets $A_{\alpha},B_{\alpha} \subseteq X_{\alpha},$ we have
\[ A_{\alpha}\delta_{\alpha} B_{\alpha} \iff A_{\alpha},B_{\alpha} \text{ are not compact in } X_{\alpha} \text{ or } cl_{X_{\alpha}}(A_{\alpha}) \cap cl_{X_{\alpha}}(B_{\alpha}) \neq \emptyset.\]
Then clearly neither the compactification of the coproduct is the one-point compactification, nor the compactification of the coproduct is the disjoint union of one point compactifications.
\end{example}

\begin{example}
Let $(X_{\alpha},d_{\alpha})_{\alpha}$ be an infinite collection of metric spaces. Equip each $X_{\alpha}$ with its metric proximity $\delta_{\alpha},$ i.e., for any two subsets $A_{\alpha},B_{\alpha} \subseteq X_{\alpha},$ we have
\[A_{\alpha}\delta_{\alpha}B_{\alpha} \iff d_{\alpha}(A_{\alpha},B_{\alpha})=0.\]
Choose a base point $x_{\alpha} \in X_{\alpha}$ for each $\alpha$. Define a metric $d$ on $\coprod_{\alpha} X_{\alpha}$ by
\[
  d(x,y) =
  \begin{cases}
    d_{\alpha}(x,y) & \text{if } x,y \in X_{\alpha} \\
d_{\alpha}(x,x_{\alpha})+d_{\beta}(x_{\beta},y)+1
& \text{if } x \in X_{\alpha}, y \in X_{\beta} \text{ and } \alpha \neq \beta
  \end{cases} 
\]
Notice that for any two subsets $A$ and $B$ of the coproduct $(\coprod_{\alpha}X_{\alpha}, \delta),$ we have that
\[A{\delta}B \implies d(A,B)=0.\]
It is easy to see that the opposite implication does not have to be true for certain collection of metric spaces. For example, take the coproduct of the real lines indexed by the positive integers and set $A:=\{1_n \mid n \in \mathbb{N}\}$ and $B:=\{1_n+\frac{1}{n} \mid n \in \mathbb{N}\}$. Then $d(A,B)=0,$ but $A{\centernot{\delta}}B.$ Thus, even though the metric $d$ agrees with $d_{\alpha}$ on $X_{\alpha}$ for all $\alpha$ and induces the same topology on the coproduct as $\delta,$ the Smirnov compactification of the coproduct $(\coprod_{\alpha}X_{\alpha}, \delta)$ can be strictly larger than the minimum uniform compactification of $\coprod_{\alpha}X_{\alpha}$ associated to the metric $d.$
Also notice that the the Smirnov compactification of the coproduct $(\coprod_{\alpha}X_{\alpha}, \delta)$ does not have to be homeomorphic to the Stone-Cech compactification of $(\coprod_{\alpha}X_{\alpha}, \delta).$ One can see that by taking the coproduct  of the real lines indexed by the positive integers as before, and noting that the proximity on $\coprod_{\alpha}X_{\alpha}$ associated to the Stone-Cech compactification is  given by: two subsets of $\coprod_{\alpha}X_{\alpha}$ are close if and only if their closures (in the topology induced by $\delta$) intersect (this follows from the fact that the disjoint union of the real lines is normal). But then by taking 
\[A:=\{n \in \mathbb{R}_1 \subseteq \coprod_{\mathbb{N}}\mathbb{R} \mid n \in \mathbb{N}\},\]
i.e., $A$ is a set of positive integers in the first real line, and 
\[B:=\{n+\frac{1}{n} \in \mathbb{R}_1 \subseteq \coprod_{\mathbb{N}}\mathbb{R} \mid n \in \mathbb{N}\},\]
we see that $A \delta B,$ but the closures of $A$ and $B$ in the topology induced by $\delta$ do not intersect. Thus, the Smirnov compactification of $(\coprod_{\mathbb{N}}\mathbb{R}, \delta)$ is strictly larger than the minimum uniform compactification of $\coprod_{\mathbb{N}}\mathbb{R}$ associated to the metric $d,$ but strictly smaller then the Stone-Cech compactification of $\coprod_{\mathbb{N}}\mathbb{R}$ associated to the topology induced by $\delta.$
\end{example}

\section{Properties of the Smirnov compactification of the coproduct}\label{properties_of_the_coproduct}
In this section, we explore some properties of the Smirnov compactification of the coproduct of separated  proximity spaces.

\subsection{Coproduct and the Stone-Cech compactification}\label{Stone-Cech}

The last example in the previous section suggests the following proposition.

\begin{proposition}\label{first_half_of_the_theorem}
Let $(X_{\alpha}, \delta_{\alpha})_{\alpha}$ be a collection of separated proximity spaces and let $(\coprod_{\alpha} X_{\alpha}, \delta)$ be the associated coproduct. If for some $\alpha$ one of the two equivalent conditions is satisfied:
\begin{enumerate}
\item the Smirnov compactification of $(X_{\alpha}, \delta_{\alpha})$ is not homeomorphic to the Stone-Cech compactification of $(X_{\alpha}, \delta_{\alpha}),$ 
\item there exist two subsets $A_{\alpha},B_{\alpha}\subseteq X_{\alpha}$ such that $A\delta_{\alpha} B,$ but $A$ and $B$ are functionally distinguishable,  
\end{enumerate}
then the Smirnov compactification of the coproduct $(\coprod_{\alpha} X_{\alpha}, \delta)$ is not homeomorphic to the Stone-Cech compactification of $(\coprod_{\alpha} X_{\alpha},\delta).$
\end{proposition}
\begin{proof}
Let $(X_{\alpha}, \delta_{\alpha})$ be as in the statement of (1). Since the proximity defined by "$A$ is close to $B$ if and only if $A$ and $B$ are functionally distinguishable" always induces the Stone-Cech compactification, it easily follows from the one-to-one correspondence of compatible proximities and compactifications that $(1)$ is equivalent to $(2).$ Thus, there exist subsets $A_{\alpha},B_{\alpha}\subseteq X_{\alpha}$ such that $A_{\alpha}\delta_{\alpha} B_{\alpha},$ but $A_{\alpha}$ and $B_{\alpha}$ are functionally distinguishable, i.e., there exists a continuous function $f_{\alpha}:X_{\alpha} \to [0,1]$ such that $f_{\alpha}(A_{\alpha})=0$ and $f_{\alpha}(B_{\alpha})=1.$  Thus, when $A$ and $B$ are considered as subsets of $\coprod_{\alpha} X_{\alpha},$ we have that $A\delta B,$ but $A$ and $B$ are functionally distinguishable by a continuous function $f: \coprod X_{\alpha} \to [0,1]$ defined by
\[
  f(x) =
  \begin{cases}
    f_{\alpha}(x) & \text{if } x \in X_{\alpha}, \\
x_0 & \text{otherwise},
  \end{cases}
\]
where $x_0$ is an arbitrary fixed point in $\coprod_{\alpha} X_{\alpha}.$ Consequently, the Stone-Cech compactification of $(\coprod_{\alpha} X_{\alpha}, \delta)$ is strictly larger than the Smirnov compactification of $(\coprod_{\alpha} X_{\alpha}, \delta).$
\end{proof}

\begin{theorem}\label{theorem2}
Let $(X_{\alpha}, \delta_{\alpha})_{\alpha}$ be a collection of separated proximity spaces and let $(\coprod_{\alpha} X_{\alpha}, \delta)$ be the associated coproduct. Then the Smirnov compactification of the coproduct $(\coprod_{\alpha} X_{\alpha}, \delta)$ is homeomorphic to the Stone-Cech compactification of $(\coprod_{\alpha} X_{\alpha},\delta)$ if and only if the Smirnov compactification of $(X_{\alpha}, \delta_{\alpha})$ is homeomorphic to the Stone-Cech compactification of $(X_{\alpha},\delta_{\alpha})$ for all $\alpha.$
\end{theorem}
\begin{proof}
The forward direction is the statement of Proposition \ref{first_half_of_the_theorem}. To prove the converse, assume that the Smirnov compactification of the coproduct $(\coprod_{\alpha} X_{\alpha}, \delta)$ is not homeomorphic to the Stone-Cech compactification of $(\coprod_{\alpha} X_{\alpha},\delta).$ Thus, there exist $A,B \subseteq \coprod_{\alpha}X_{\alpha}$ such that $A\delta B,$ but $A$ and $B$ are functionally distinguishable by a function $f:\coprod_{\alpha}X_{\alpha} \to [0,1].$ In particular, there exists $\alpha$ such that $A_{\alpha}\delta_{\alpha}B_{\alpha}$ and $A_{\alpha}$ and $B_{\alpha}$ are functionally distinguishable by a restriction of $f$ to $X_{\alpha}.$ Thus, the Smirnov compactification of $(X_{\alpha}, \delta_{\alpha})$ is not homeomorphic to the Stone-Cech compactification of $(X_{\alpha},\delta_{\alpha}).$ 
\end{proof}

\begin{corollary}
Let $(X_{\alpha})_{\alpha}$ be a collection of completely regular Hausdorff spaces. Then the Stone-Cech compactification of the disjoint union topology of the collection $(X_{\alpha})_{\alpha}$ is homeomorphic to the disjoint union topology of the Stone-Cech compactifications of all $X_{\alpha}$'s if and only if the index set is finite.
\end{corollary}
\begin{proof}
This is an immediate consequence of Theorem \ref{theorem1} and Theorem \ref{theorem2}.
\end{proof}

\begin{corollary}
Let $(X_{\alpha}, \delta_{\alpha})_{\alpha}$ be a collection of separated proximity spaces where each $\delta_{\alpha}$ is the Stone-Cech proximity, and where each $\delta_{\alpha}$ induces a normal topology on $X_{\alpha}.$ Let $(\coprod_{\alpha} X_{\alpha}, \delta)$ be the associated coproduct. Then 
\begin{enumerate}
\item $\delta$ is equinormal, i.e., $\delta$-close subsets of  $\coprod_{\alpha} X_{\alpha}$ have intersecting closures in $\coprod_{\alpha} X_{\alpha}.$
\item The proximity $\delta$ on $\coprod_{\alpha} X_{\alpha}$ can be alternatively described by
\[A \delta B \iff \overline{A} \cap \overline{B} \neq \emptyset,\]
where $\overline{A}$ means the closure of $A$ in $\coprod_{\alpha} X_{\alpha}.$
\item every real-valued continuous function on $\coprod_{\alpha} X_{\alpha}$ is a proximity map.
\end{enumerate}
\end{corollary}
\begin{proof}
By Theorem \ref{theorem2}, the Smirnov compactification of the coproduct is the Stone-Cech compactification. Since the topology on each $X_{\alpha}$ is normal, so is the topology on the coproduct. Consequently, by Corollary 7.23 in \cite{proximityspaces}, $\delta$ is equinormal. This shows $(1).$ To see $(2),$ notice that 
\[\overline{A} \cap \overline{B} \neq \emptyset \implies \overline{A} \delta \overline{B} \iff A\delta B,\]
where the first implication follows from the definition of a proximity and the second equivalence is the statement of Lemma 2.8 in \cite{proximityspaces}. Notice that if the proximity is equinormal, then the first implication is an equivalence. Since $\delta$ is equinormal, the conclusion follows. Finally, $(3)$ is equivalent to being equinormal on normal separated proximity spaces by Theorem 7.22 in \cite{proximityspaces}.
\end{proof}

For more on equinormal proximity spaces, the reader is referred to \cite{proximityspaces} and \cite{Pervin}. 

\subsection{Metrizability and proximity weight}\label{metrizability}

The following theorem shows that the boundary of the Smirnov compactification of an infinite coproduct always contains the Stone-Cech corona of the naturals. The proof resembles the technique used to show that metric spaces which are not totally bounded have a non-metrizable minimal uniform compactification (see for example Theorem 3.3 in \cite{woods}). To understand the statement of the proposition, recall that $\omega$ denotes a countably infinite discrete space and $\beta \omega$ denotes the Stone-Cech compactification of $\omega.$

\begin{theorem}\label{theorem123}
Let $(X_{\alpha}, \delta_{\alpha})_{\alpha}$ be an infinite collection of separated proximity spaces, $(\coprod_{\alpha} X_{\alpha}, \delta)$ be the associated coproduct, and $(\overline{\coprod_{\alpha} X_{\alpha}}, \delta^*)$ be the associated Smirnov compactification. Then  $(\overline{\coprod_{\alpha} X_{\alpha}})\setminus (\coprod_{\alpha} X_{\alpha})$ contains a copy of $\beta \omega \setminus \omega.$ In particular, the topology on $\overline{\coprod_{\alpha} X_{\alpha}}$ is not metrizable. Consequently, the proximity on $\overline{\coprod_{\alpha} X_{\alpha}}$ is not a metric proximity for any metric. 
\end{theorem}
\begin{proof}
Since $(X_{\alpha}, \delta_{\alpha})_{\alpha}$ is an infinite collection, it contains a countably infinite subset. Choose one point from each element in that countably infinite subset. Denote that collection of points by $D$. Then $D \subseteq \coprod_{\alpha} X_{\alpha}.$ For any two disjoint subsets $A$ and $B$ of $D$ we have that $cl_{\overline{\coprod_{\alpha} X_{\alpha}}}(A) \cap cl_{\overline{\coprod_{\alpha} X_{\alpha}}}(A) = \emptyset$ (because if their closures in $\overline{\coprod_{\alpha} X_{\alpha}}$ would intersect, then $A \delta B$, which is not true). Since the closure of $D$ in $\overline{\coprod_{\alpha} X_{\alpha}}$ can be identified with the Smirnov compactification of $D$ (when $D$ is given the subspace proximity inherited from $\coprod_{\alpha} X_{\alpha}$), we see that any two disjoint subsets of $D$ have disjoint closures in the Smirnov compactification of $D$. Thus, by Theorem 6.5 in \cite{Rings}, we know that the closure of $D$ in $\overline{\coprod_{\alpha} X_{\alpha}}$ (i.e., the Smirnov compactification of $D$) is homeomorphic to the Stone-Cech compactification of $D,$ i.e.,
\[cl_{\overline{\coprod_{\alpha} X_{\alpha}}}(D) \cong \mathfrak{D} \cong  \beta \omega,\]
where $\cong$ means homeomorphic, and $\mathfrak{D}$ denotes the Smirnov compactification of $D.$ Since $D$ is closed in $\coprod_{\alpha} X_{\alpha},$ we have that
\[(\mathfrak{D} \setminus D) \subseteq (\overline{\coprod_{\alpha} X_{\alpha}}\setminus \coprod_{\alpha} X_{\alpha}).\]
Thus,
\[\beta \omega \setminus \omega \cong (\mathfrak{D} \setminus D) \subseteq (\overline{\coprod_{\alpha} X_{\alpha}}\setminus \coprod_{\alpha} X_{\alpha}). \]
In particular, since $\beta \omega \setminus \omega$ is not metrizable, neither is $\overline{\coprod_{\alpha} X_{\alpha}}.$ Thus, the proximity $\delta^*$ on $\overline{\coprod_{\alpha} X_{\alpha}}$ is not a metric proximity (otherwise, that metric would induce the same topology as $\delta$ on $\overline{\coprod_{\alpha} X_{\alpha}},$ and thus the topology on $\overline{\coprod_{\alpha} X_{\alpha}}$ would be metrizable).
\end{proof}

To understand the following corollary, recall that the \textbf{proximity weight} of a proximity space $(X, \delta)$ is the smallest cardinal number $\kappa$ such that $X$ has a proximity base $\mathcal{P}$ with $|\mathcal{P}|\leq \kappa,$ where by the \textbf{proximity base} we mean a collection $\mathcal{P}$ of subsets of $X$ such that $A {\centernot{\delta}}B$ implies that there are $C,D \subseteq \mathcal{P}$ such that $A\subseteq C, B \subseteq D,$ and $C{\centernot{\delta}}D.$

\begin{corollary}\label{corollary123}
Let $(X_{\alpha}, \delta_{\alpha})_{\alpha}$ be an infinite collection of separated proximity spaces, $(\coprod_{\alpha} X_{\alpha}, \delta)$ be the associated coproduct, and $(\overline{\coprod_{\alpha} X_{\alpha}}, \delta^*)$ be the associated Smirnov compactification. Then the proximity weights of $(\coprod_{\alpha} X_{\alpha}, \delta)$ and $(\overline{\coprod_{\alpha} X_{\alpha}}, \delta^*)$ are strictly larger than the cardinality of the natural numbers.
\end{corollary}
\begin{proof}
First recall that by Theorem 8.14 in \cite{proximityspaces}, given a proximity space $X$ and its Smirnov compactification $\mathfrak{X},$ the proximity weights of $X$ and $\mathfrak{X}$ agree. Thus, it is enough to show that the proximity weight of $(\overline{\coprod_{\alpha} X_{\alpha}}, \delta^*)$ is larger than the cardinality of the natural numbers. However, by Theorem 8.19 in \cite{proximityspaces}, if $(\overline{\coprod_{\alpha} X_{\alpha}}, \delta^*)$ had the proximity weight equal to natural numbers, then $(\overline{\coprod_{\alpha} X_{\alpha}}, \delta^*)$ would be metrizable (even totally bounded), which is a contradiction by Theorem \ref{theorem123}.
\end{proof}

\begin{corollary}
Let $(X_{\alpha}, \delta_{\alpha})_{\alpha}$ be an infinite collection of separated proximity spaces and $(\coprod_{\alpha} X_{\alpha}, \delta)$ be the associated coproduct. Then $\delta$ is not a metric proximity for any totally bounded metric on $\coprod_{\alpha} X_{\alpha}.$
\end{corollary}
\begin{proof}
If $\delta$ was a metric proximity for some totally bounded metric $d$ on $\coprod_{\alpha} X_{\alpha},$ then $d$ and $\delta$ would induce the same topology on $\coprod_{\alpha} X_{\alpha}.$ But by theorem 8.19 in \cite{proximityspaces}, this would mean that the proximity weight of $(\coprod_{\alpha} X_{\alpha}, \delta)$ equals the cardinality of the natural numbers, which contradicts Corollary \ref{corollary123}.
\end{proof}

\subsection{Connectedness of the boundary}\label{connectedness}

The following theorem shows that the boundary of the Smirnov compactification of a coproduct is an example of a highly disconnected compact space.

\begin{theorem}\label{connectedness_theorem}
Let $(X_{\alpha}, \delta_{\alpha})_{\alpha}$ be a collection of separated proximity spaces, $(\coprod_{\alpha} X_{\alpha}, \delta)$ be the associated coproduct, and $(\overline{\coprod_{\alpha} X_{\alpha}}, \delta^*)$ be the associated Smirnov compactification. Then $(\overline{\coprod_{\alpha} X_{\alpha}})\setminus (\coprod_{\alpha} X_{\alpha})$ has at least as many 
subsets that are simultaneously open and closed as there are elements in the index set. In particular, $(\overline{\coprod_{\alpha} X_{\alpha}})\setminus (\coprod_{\alpha} X_{\alpha})$ has at least as many connected components as there are elements in the index set.
\end{theorem}
\begin{proof}
For any $\alpha,$ let $X_{\alpha}^*$ denote the intersection of the closure of $X_{\alpha}$ in $\overline{\coprod_{\alpha} X_{\alpha}}$ with the boundary of the Smirnov compactification of the coproduct, i.e.,
\[X_{\alpha}^*:=\Big( cl_{\overline{\coprod_{\alpha} X_{\alpha}}}(X_{\alpha})\Big) \cap \Big((\overline{\coprod_{\alpha} X_{\alpha}})\setminus (\coprod_{\alpha} X_{\alpha})\Big).\]
Clearly $X_{\alpha}^*$ is closed in $(\overline{\coprod_{\alpha} X_{\alpha}})\setminus (\coprod_{\alpha} X_{\alpha}),$ being the intersection of a closed set in $\overline{\coprod_{\alpha}X_{\alpha}}$ with $(\overline{\coprod_{\alpha} X_{\alpha}})\setminus (\coprod_{\alpha} X_{\alpha}).$ Notice that for $\alpha \neq \beta,$ we have that $X_{\alpha}^* \cap Y_{\beta}^* = \emptyset,$ since otherwise \[cl_{\overline{\coprod_{\alpha} X_{\alpha}}} (X_{\alpha}) \cap cl_{\overline{\coprod_{\alpha} X_{\alpha}}} (X_{\beta}) \neq \emptyset, \]
implying that $X_{\alpha} \delta X_{\beta},$ a contradiction. Let us show that for any $\alpha$ we have that $X_{\alpha}^*$ is also open in $(\overline{\coprod_{\alpha} X_{\alpha}})\setminus (\coprod_{\alpha} X_{\alpha})$. Let $\alpha$ be arbitrary. By Corollary 2.4 in \cite{proximityspaces}, it is enough to show that for any $\sigma \in X_{\alpha}^*,$ we have that
\begin{equation}\label{important_relation}
\sigma{\centernot{\delta}}^*\Big(\big((\overline{\coprod_{\alpha} X_{\alpha}})\setminus (\coprod_{\alpha} X_{\alpha})\big)\setminus X_{\alpha}^*\Big)\tag{$*$}.
\end{equation}
Let $\sigma \in X_{\alpha}^*$ be arbitrary. To show (\ref{important_relation}), we need to find two subsets of $\coprod_{\alpha} X_{\alpha}$ such that one of them absorbs $\sigma,$ the other one absorbs $ \Big(\big((\overline{\coprod_{\alpha} X_{\alpha}})\setminus (\coprod_{\alpha} X_{\alpha})\big)\setminus X_{\alpha}^*\Big),$ but the two subsets are not $\delta$-close.
Notice that 
\begin{enumerate}[(1)]
\item $\sigma \in X_{\alpha}^*$ implies that $X_{\alpha} \in \sigma$ (i.e., $X_{\alpha}$ absorbs $\sigma$),
\item $\Big( (\coprod_{\alpha} X_{\alpha})\setminus X_{\alpha}\Big)$ absorbs $\Big(\big((\overline{\coprod_{\alpha} X_{\alpha}})\setminus (\coprod_{\alpha} X_{\alpha})\big)\setminus X_{\alpha}^*\Big).$ 
\end{enumerate}
(1) follows from the fact that given a subset $C$ of a proximity space $X,$ the closure of $C$ in the Smirnov compactification of $X$ consists of all the clusters in the Smirnov compactification of $X$ than contain $C$. To see (2), we need to show that for an arbitrary cluster $\sigma_1 \in \Big(\big((\overline{\coprod_{\alpha} X_{\alpha}})\setminus (\coprod_{\alpha} X_{\alpha})\big)\setminus X_{\alpha}^*\Big)$ we have that $\Big( (\coprod_{\alpha} X_{\alpha})\setminus X_{\alpha}\Big) \in \sigma_1.$ Notice that $X_{\alpha} \notin \sigma_1,$ for if $X_{\alpha} \in \sigma_1,$ then $\sigma_1 \in cl_{\overline{\coprod_{\alpha} X_{\alpha}}}(X_{\alpha}),$ and since $\sigma_1 \notin \coprod_{\alpha} X_{\alpha},$ this shows that $\sigma_1 \in X_{\alpha}^*,$ a contradiction to $\sigma_1 \in \Big(\big((\overline{\coprod_{\alpha} X_{\alpha}})\setminus (\coprod_{\alpha} X_{\alpha})\big)\setminus X_{\alpha}^*\Big).$ Thus, since $X_{\alpha} \notin \sigma_1,$ it has to be that $\Big( (\coprod_{\alpha} X_{\alpha})\setminus X_{\alpha}\Big) \in \sigma_1$ (this follows from a property of clusters mentioned right before Proposition \ref{generating_cluster}). Therefore, (1) and (2) hold, i.e., $X_{\alpha}$ absorbs $\sigma$ and $\Big( (\coprod_{\alpha} X_{\alpha})\setminus X_{\alpha}\Big)$ absorbs $\Big(\big((\overline{\coprod_{\alpha} X_{\alpha}})\setminus (\coprod_{\alpha} X_{\alpha})\big)\setminus X_{\alpha}^*\Big).$ Also notice that $X_{\alpha} {\centernot{\delta}}\Big( (\coprod_{\alpha} X_{\alpha})\setminus X_{\alpha}\Big)$ by the definition of $\delta.$ This finishes the proof of (\ref{important_relation}). Consequently,
$X_{\alpha}^*$ is open in $(\overline{\coprod_{\alpha} X_{\alpha}})\setminus (\coprod_{\alpha} X_{\alpha}).$ In conclusion, for any $\alpha$ we have that $X_{\alpha}^*$ is open and closed in $(\overline{\coprod_{\alpha} X_{\alpha}})\setminus (\coprod_{\alpha} X_{\alpha}),$ and $X_{\alpha}^* \cap X_{\beta}^* = \emptyset$ whenever $\alpha \neq \beta.$ Thus, $(\overline{\coprod_{\alpha} X_{\alpha}})\setminus (\coprod_{\alpha} X_{\alpha})$ has at least as many 
subsets that are simultaneously open and closed as there are elements in the index set. Since any susbet that is both open and closed is a union of connected components, $(\overline{\coprod_{\alpha} X_{\alpha}})\setminus (\coprod_{\alpha} X_{\alpha})$ has at least as many connected components as there are elements in the index set.
\end{proof}

In fact, the proof of the above theorem shows the following.
\begin{corollary}
Let $(X_{\alpha}, \delta_{\alpha})_{\alpha}$ be a collection of separated proximity spaces, $(\coprod_{\alpha} X_{\alpha}, \delta)$ be the associated coproduct, and $(\overline{\coprod_{\alpha} X_{\alpha}}, \delta^*)$ be the associated Smirnov compactification. Then for any $\alpha,$ the trace of $X_{\alpha}$ is an open and closed subset of the boundary, i.e.,
\[X_{\alpha}^*:=\Big( cl_{\overline{\coprod_{\alpha} X_{\alpha}}}(X_{\alpha})\Big) \cap \Big((\overline{\coprod_{\alpha} X_{\alpha}})\setminus (\coprod_{\alpha} X_{\alpha})\Big)\]
is open and closed in $(\overline{\coprod_{\alpha} X_{\alpha}})\setminus (\coprod_{\alpha} X_{\alpha}).$ \qed
\end{corollary}

\subsection{Dimension}\label{dimension}

Recall from \cite{Engelking} that a nonempty topological space $X$ is said to have the \textbf{covering dimension}  $dim(X)=n<\infty$ if $n$ is the smallest non-negative integer with the property that each finite open cover of $X$ has a finite open refinement of multiplicity at most $n+1.$ If no such integer exists, we say that $dim(X)=\infty.$  By the \textbf{multiplicity} of a cover, we mean the smallest number $d$ such that each $x$ in $X$ is contained in at most $d$ elements of that cover.

In \cite{proximitydimension}, Smirnov defined a similar notion of dimension for proximity spaces. In particular, a nonempty proximity space $(X,\delta)$ is said to have the \textbf{$\delta$-dimension} $\delta d (X)=n<\infty$ (also known as the \textbf{proximity dimension}) if $n$ is the smallest non-negative integer with the property that each $\delta$-covering of $X$ has a refinement that is a $\delta$-covering of multiplicity at most $n+1.$ If no such integer exists, we say that $\delta d(X)=\infty.$ By a \textbf{$\delta$-covering}, we mean a finite collection of subsets $A_1,...,A_k\subseteq X$ such that there is a collection of subsets $B_1,...,B_k\subseteq X$ with $\cup_{i\leq k}B_i=X$ and $B_i {\centernot{\delta}} (X \setminus A_i)$ (often abbreviated by $B_i\ll A_i$) for all $1\leq i \leq k.$ In particular, it easily follows that $B_i \subseteq A_i$ for all $1\leq i \leq k.$

\begin{lemma}\label{lemma14}
Let $(X,\delta)$ be a nonempty proximity space such that $\delta d (X)\leq d$ for some non-negative integer $d.$ Given a $\delta$-covering $U_1,\cdots, U_n$ of $X$ there exists a refinement $A_1,...,A_n$ (some of them possibly empty) that is a $\delta$-covering of multiplicity at most $d+1.$
\end{lemma}
\begin{proof}
Since $\delta d (X)\leq d,$ there exists a refinement $\mathcal{V}:=\{V_1,\cdots, V_m\}$ of $U_1, \cdots, U_n$ for some non-negative integer $m$ such that $V_1,\cdots, V_m$ is a $\delta$-covering of multiplicity at most $d+1.$ If $m\leq n,$ then we are done, so assume that $m>n.$ since $V_1,\cdots, V_m$ is a $\delta$-covering, there exists $\mathcal{B}:=\{B_1,...,B_m\}$ such that $\cup_{i\leq m} B_i=X$ and $B_i \ll V_i$ for all $1\leq i \leq m.$ Define inductively for all $2\leq i\leq n$
\[A_1:=\bigcup_{\{V \in \mathcal{V} \mid V \in U_1\}}V,\]
\[A_i:=\bigcup_{\{V \in \mathcal{V} \mid V \in U_i \text{ and } V \notin A_j \text{ for all } j<i\}} V.\]
Clearly $A_1,\cdots,A_n$ has multiplicity at most $d+1,$ since each $A_i$ consists of unions of elements of $\mathcal{V}$ and each element of $\mathcal{V}$ is in at most one $A_i.$ Let us show that $A_1,...,A_n$ is a $\delta$-cover, which will finish the proof. For all $1 \leq i \leq n,$ set $C_i$ to be the finite union of $B$'s associated to the $V's$ in $A_i,$ i.e
\[C_i:=\bigcup_{\{B_j \in \mathcal{B} \mid B_j\ll V_j \text{ for some } V_j \in A_i\}}B.\]
In particular, each $C_i$ is the union of as many elements of $\mathcal{B}$ as there are elements in the union that defines $A_i.$ To see that $C_1,\cdots, C_n$ covers $X,$ let $x \in X$ be arbitrary. Then $x \in B_j$ for at least one $j$ such that $1\leq j\leq m.$ Let $i$ be the smallest integer such that 
\[x\in B_k \ll V_k \subseteq U_i,\]
for some integer $k$ such that $1\leq k \leq m.$ Then clearly $V_k \subseteq A_i,$ and thus $B_k \subseteq C_i.$ Thus, $x \in C_i,$ which shows that  $C_1,\cdots, C_n$ covers $X.$ Finally, let $i$ be an  arbitrary integer such that $1\leq i \leq n$. Then
\[A_i=V_{j_1}\cup V_{j_2}\cup \cdots \cup V_{j_s}\]
\[C_i=B_{j_1}\cup B_{j_2}\cup \cdots \cup B_{j_s}\]
for some nonnegative integer $s.$ Since we have
\[B_{jl} \ll V_{jl} \text{ for all } 1\leq l \leq s,\]
by Corlollary 3.10 in \cite{proximityspaces}, we have that $C_i \ll A_i,$ as desired.
\end{proof}

\begin{theorem}\label{last_theorem}
Let $(X_{\alpha}, \delta_{\alpha})_{\alpha}$ be a collection of nonempty separated proximity spaces,  
 $(\coprod_{\alpha} X_{\alpha}, \delta)$ be the associated coproduct, and $(\overline{\coprod_{\alpha} X_{\alpha}}, \delta^*)$ be the associated Smirnov compactification. For any $\alpha,$ let $(\mathfrak{X}_{\alpha},\delta_{\alpha}^*)$ denote the Smirnov compactification of $(X_{\alpha}, \delta_{\alpha}),$ and let $(\coprod_{\alpha} \mathfrak{X}_{\alpha}, \delta_1)$ denote the coproduct of Smirnov compactifications. Then
\begin{multline*}
\delta d \left(\coprod_{\alpha} X_{\alpha}\right)= \delta_1 d \left(\coprod_{\alpha} \mathfrak{X}_{\alpha}\right) =\delta^* d \left(\overline{\coprod_{\alpha} X_{\alpha}}\right)=dim\left(\overline{\coprod_{\alpha} X_{\alpha}}\right) \\ =\sup_{\alpha} dim\left(\mathfrak{X}_{\alpha}\right)=\sup_{\alpha} \delta_{\alpha}d\left(X_{\alpha}\right)=\sup_{\alpha} \delta_{\alpha}^*d\left(\mathfrak{X}_{\alpha}\right).
\end{multline*}
\end{theorem}
\begin{proof}
Since $(\coprod_{\alpha} X_{\alpha}, \delta)$ proximally and densely embedds in its Smirnov compactification $(\overline{\coprod_{\alpha} X_{\alpha}}, \delta^*),$ by Theorem $3$ in \cite{proximitydimension}, we know that $\delta d \left(\coprod_{\alpha} X_{\alpha}\right)=\delta^* d\left(\overline{\coprod_{\alpha} X_{\alpha}}\right).$ Since by Theorem \ref{theorem1} we also know that $(\coprod_{\alpha} \mathfrak{X}_{\alpha}, \delta_1)$ proximally and densely embedds in $(\overline{\coprod_{\alpha} X_{\alpha}}, \delta^*),$ we have that $\delta_1 d \left(\coprod_{\alpha} \mathfrak{X}_{\alpha}\right) =\delta^* d \left(\overline{\coprod_{\alpha} X_{\alpha}}\right).$ By theorem $1$ in \cite{proximitydimension}, the $\delta$-dimension of a proximity space coincides with the topological dimension of its Smirnov extension. Thus, $\delta^* d \left(\overline{\coprod_{\alpha} X_{\alpha}}\right)=dim\left(\overline{\coprod_{\alpha} X_{\alpha}}\right).$ For the same reason, we have that $\sup_{\alpha} dim\left(\mathfrak{X}_{\alpha}\right)=\sup_{\alpha} \delta_{\alpha}d\left(X_{\alpha}\right).$ Finally, again by Theorem $3$ in \cite{proximitydimension}, we have that any proximity space has the same $\delta$-dimension as its Smirnov compactification, and consequently $\sup_{\alpha} \delta_{\alpha}d\left(X_{\alpha}\right)=\sup_{\alpha} \delta_{\alpha}^*d\left(\mathfrak{X}_{\alpha}\right).$ Therefore, to finish the proof it is enough to show that $\delta_1 d \left(\coprod_{\alpha} \mathfrak{X}_{\alpha}\right) =\sup_{\alpha} \delta_{\alpha}^*d\left(\mathfrak{X}_{\alpha}\right).$ 

Since $\mathfrak{X}_{\alpha}$ is a proximity subspace of $\coprod_{\alpha} \mathfrak{X}_{\alpha}$ for all $\alpha,$ Theorem $2$ in \cite{proximitydimension} implies that
\[\delta_1 d \left(\coprod_{\alpha} \mathfrak{X}_{\alpha}\right) \geq \sup_{\alpha} \delta_{\alpha}^*d\left(\mathfrak{X}_{\alpha}\right).\]
To see that $\delta_1 d \left(\coprod_{\alpha} \mathfrak{X}_{\alpha}\right) \leq \sup_{\alpha} \delta_{\alpha}^*d\left(\mathfrak{X}_{\alpha}\right),$ suppose that $\sup_{\alpha} \delta_{\alpha}^*d\left(\mathfrak{X}_{\alpha}\right)=d,$ since the inequality is clear when $\sup_{\alpha} \delta_{\alpha}^*d\left(\mathfrak{X}_{\alpha}\right)=\infty$. Take any $\delta_1$-covering $A^1,A^2,\dots ,A^k$ of $\left(\coprod_{\alpha} \mathfrak{X}_{\alpha}\right).$ Notice that for each $\alpha,$ the restriction of this covering to $\mathfrak{X}_{\alpha}$ is a $\delta_{\alpha}$-covering of $X_{\alpha},$ i.e., $A^1_{\alpha}, A^2_{\alpha}, \dots, A^k_{\alpha}$ is a $\delta_{\alpha}$-covering of $X_{\alpha}.$ Since $\delta_{\alpha}d (\mathfrak{X}_{\alpha})\leq d,$ by Lemma \ref{lemma14}, we know that there exists a finite refinement $A^1_{\alpha},...,A^k_{\alpha}$ (some of them possibly empty) that is a $\delta_{\alpha}$-covering of $\mathfrak{X}_{\alpha}$ of multiplicity at most $d+1.$ Then for all $i$ such that $1\leq i\leq k,$ define
\[A_i:=\bigcup_{\alpha}A^i_{\alpha}.\]
By construction, it is easy to see that $A_1, \cdots, A_k$ is a finite $\delta_1$-covering of $\coprod_{\alpha}\mathfrak{X}_{\alpha}$ of multiplicity at most $d+1.$ Thus, $\delta_1 d \left(\coprod_{\alpha} \mathfrak{X}_{\alpha}\right) \leq d =\sup_{\alpha} \delta_{\alpha}^*d\left(\mathfrak{X}_{\alpha}\right),$ finishing the proof that $\delta_1 d \left(\coprod_{\alpha} \mathfrak{X}_{\alpha}\right) = \sup_{\alpha} \delta_{\alpha}^*d\left(\mathfrak{X}_{\alpha}\right).$
\end{proof}

\begin{example}
Given a subset $A\subseteq \mathbb{R}^n,$ define
\[A_{\mathbb{Q}}:=A \cap \mathbb{Q}^n,\]
where $\mathbb{Q}$ denotes the rational numbers. For each positive integer $n,$ let 
\[X_n:=([0,1]^n)_{\mathbb{Q}}.\]
Equip each $X_n$ with the proximity $\delta_{n}$ corresponding to the compactification $[0,1]^n.$ Consider the coproduct $(\coprod_{n}X_{n}, \delta),$ i.e.,
 \[\coprod_{n}X_{n}=[0,1]_{\mathbb{Q}} \sqcup \left( [0,1] \times [0,1] \right)_{\mathbb{Q}} \sqcup \left( [0,1] \times [0,1] \times [0,1] \right)_{\mathbb{Q}} \sqcup \cdots\]
with the coproduct proximity $\delta.$
By theorem \ref{last_theorem}, the proximity dimension of the coproduct $\coprod_{\alpha}X_{\alpha}$ equals the supremum of the covering dimensions of the Smirnov compactifications of the original proximity spaces, i.e., the proximity dimension of the coproduct is $\infty.$ Consequently, the covering dimension of the Smirnov compactification of the coproduct is also $\infty.$ This happens despite the fact that each $X_{\alpha}$ has covering dimension 0. In fact, $\coprod_{n}X_{n}$ also has covering dimension $0.$ This follows from the fact that $\coprod_{n}X_{n}$ has countable cardinality and is metrizable (see for example $1.2.5$ in \cite{Engelking}). Thus, we have
\[\delta d(\coprod_{n}X_{n})= \infty \quad \quad \text{and} \quad \quad dim(\coprod_{n}X_{n})=0.\]
\end{example}

\bibliographystyle{abbrv}
\bibliography{coproducts_of_proximity_spaces}{}

\end{document}